\documentclass{article}

\usepackage{amssymb}

\usepackage{amsmath}
\usepackage{amsthm}

\usepackage[shortlabels]{enumitem}

\numberwithin{equation}{section}

\newcommand{\qua}{\quad}

\newcommand{\Hejnice}[1]{$<\!\!{#1}$-transitively Marczewski null}
\newcommand{\leqHejnice}[1]{$\leq\!\!{#1}$-transitively Marczewski null}

\theoremstyle{plain}

\newtheorem{thm}{Theorem}
\newtheorem{cor}[thm]{Corollary}
\newtheorem{ques}[thm]{Question}
\newtheorem{lem}[thm]{Lemma}
\newtheorem{obs}[thm]{Observation}
\newtheorem{claim}[thm]{Claim}

\newtheorem*{Taylor-cor}{Corollary to Taylor's Theorem}
\newtheorem*{Erdoes-Rado}{Erd\H os-Rado Theorem}

\newtheorem{propos}[thm]{Proposition}

\theoremstyle{definition}
\newtheorem{defi}[thm]{Definition}

\theoremstyle{definition}

\theoremstyle{remark}

\newcounter{enuroman}
\renewcommand{\theenuroman}{\roman{enuroman}}
\newenvironment{romanenumerate}{\begin{list}{\rm (\theenuroman)}{\usecounter{enuroman}
    \setlength{\labelwidth}{1cm}}}
   {\end{list}}

\newcounter{enuRoman}
\renewcommand{\theenuRoman}{(\Roman{enuRoman})}

\newcounter{enuAlph}
\renewcommand{\theenuAlph}{\Alph{enuAlph}}

\newcounter{enualph}
\renewcommand{\theenualph}{\alph{enualph}}

\newcounter{enuarabic}
\renewcommand{\theenuarabic}{\arabic{enuarabic}}

\newcommand{\forces}{\Vdash}

\newcommand{\rest}{{\upharpoonright}}
\newcommand{\restrict}{{\upharpoonright}}

\newcommand{\I}{{\mathcal I}}

\newcommand{\M}{{\mathcal M}}
\newcommand{\N}{{\mathcal N}}

\renewcommand{\P}{{\mathcal P}}

\newcommand{\T}{{\mathcal T}}

\newcommand{\CC}{{\mathbb C}}

\newcommand{\PP}{{\mathbb P}}

\newcommand{\RR}{{\mathbb R}}
\renewcommand{\SS}{{\mathbb S}}

\newcommand{\bb}{{\mathfrak b}}

\newcommand{\cc}{{\mathfrak c}}

\newcommand{\cov}{{\mathsf{cov}}}

\newcommand{\cof}{{\mathsf{cof}}}

\newcommand{\cf}{{\mathrm{cf}}}

\newcommand{\sub}{\subseteq}
\newcommand{\sem}{\setminus}

\newcommand{\twoom}{{}^\omega 2}

\newcommand{\twolom}{{}^{<\omega} 2}

\newcommand{\omom}{{}^\omega \omega}

\newcommand{\la}{\langle}
\newcommand{\ra}{\rangle}

\newcommand{\rrr}{r}

\newcommand{\muu}{\mu}

\newcommand{\predec}[2]{{pred(#1,#2)}}

\title{Borel Conjecture for the Marczewski ideal}

\author{J\"org Brendle \\
Graduate School of System Informatics \\
   Kobe University \\
   Rokko-dai 1-1, Nada-ku \\
   Kobe 657-8501, Japan \\
email: {\sf brendle@kobe-u.ac.jp} \\
\\
Wolfgang Wohofsky \\
Hamburg University \\
   Bundesstrasse 55 (Geomatikum) \\
   20146 Hamburg, Germany \\
email: {\sf wolfgang.wohofsky@gmx.at}}

\begin{document}
\maketitle



\begin{abstract}
We show in ZFC that there is no set
of reals
of size
continuum
which
can be translated away from every set in the Marczewski ideal.
We also show that in the Cohen model, every set with this property is countable.
\end{abstract}

\section{Introduction}\label{sec:Introduction}

\subsection{Basic definitions and historical remarks}

The results of this paper are concerned with translations of certain
special
sets of real numbers:
for technical reasons,
we are
mainly
going to work
with
the Polish group
$(\twoom, +)$
as
``the reals'',
where
$\twoom$ is the
well-known
Cantor space, and
$+$ is bitwise addition
modulo $2$;
in Section~\ref{sec:arbitrary_Polish_groups}, however, we are going to show how to generalize
the main ZFC result
to arbitrary Polish groups; in particular,
it
also
holds
for the classical real line $\RR$, the circle $S^1$, etc.

Let
$\M \sub \P(\twoom)$ and $\N \sub \P(\twoom)$
denote
the $\sigma$-ideals of
\emph{meager} sets and
\emph{(Lebesgue) measure zero} sets, respectively.

A set $X \sub \twoom$ is \emph{strong measure zero}
if for each sequence $(\varepsilon_n)_{n \in \omega}$
of positive real numbers
there is a sequence $(B_n)_{n \in \omega}$
such that for each $n \in \omega$, the diameter of
$B_n \sub \twoom$
is less than $\varepsilon_n$, and $X \sub \bigcup_{n \in \omega} B_n$.
(See~\cite[Chapter~8]{BartoszynskiJudah}
for more information on strong measure zero sets and related concepts.)

For
$X, Y \sub \twoom$ and $t \in \twoom$,
let $X + Y := \{ x + y : x \in X, y \in Y \}$,
and $X + t := \{ x + t : x \in X \}$.
Let $\I \sub \P(\twoom)$ be any collection,
e.g., an ideal such as $\M$ or $\N$.

\begin{defi}
A set $X \sub \twoom$ is
\emph{$\I$-shiftable}
if $X + Y \neq \twoom$ for each $Y \in \I$.
\end{defi}

It is easy to
check
that
a set $X$ is
$\I$-shiftable
if
and only if
it can be ``translated away'' from
each
set in $\I$
(i.e., for every $Y \in \I$ there is a $t \in \twoom$ such that $(X + t) \cap Y = \emptyset$); therefore the name.
Note that each countable set is $\I$-shiftable whenever $\I$ is a $\sigma$-ideal.
Moreover, the $\I$-shiftable sets are clearly closed under subsets, so they are a
``notion of smallness''. However, they do not necessarily form an ideal, even in ``nice'' cases such as $\I = \N$ (see~\cite{BrSh:607}).

The Galvin-Mycielski-Solovay theorem
(see~\cite{GMS})
says
that a set is strong measure zero if and only if it is $\M$-shiftable.
This gives rise to new
notions of smallness,
by replacing $\M$ by other collections.
The first example
is the following notion ``dual to strong measure zero'',
which was introduced by P\v{r}ikry:
a set $X \sub \twoom$ is \emph{strongly meager}
if it is $\N$-shiftable.

Since $\M$ and $\N$ are $\sigma$-ideals, each countable set is both strong measure zero and strongly meager.
Assuming CH (the ``Continuum Hypothesis'', i.e., $2^{\aleph_0} = \aleph_1$), there are uncountable strong measure zero
sets
(shown in~\cite{SierpinskiCH})
and uncountable strongly meager sets;
this basically follows
(via a straightforward inductive construction)
from the fact that $\M$ and $\N$ have Borel bases,
hence
bases of size continuum
(i.e., $\cof(\M) \leq \cc$ and $\cof(\N) \leq \cc$, where the \emph{cofinality $\cof(\I)$} is the smallest size of a basis of $\I$).
On the other hand,
a famous result by Laver (see~\cite{LaverConBC}) shows that
it is consistent with ZFC that the \emph{Borel Conjecture} (BC) holds:
this is the
statement
that each strong measure zero set is countable
(conjectured in~\cite{Borel});
the \emph{dual Borel Conjecture} (dBC),
i.e., the statement that each strongly meager set is countable,
is consistent with ZFC as well (as shown by Carlson, see~\cite{CarlsonCondBC});
moreover, it is consistent that BC and dBC hold simultaneously
(see~\cite{GoKrShWo:969}).
For a survey on these topics, see also~\cite{SurveyArticle}.

The purpose of this paper
is to explore
the collection of $s_0$-shiftable sets,
where $s_0$ is the Marczewski ideal (see~\cite{Miller_special}):
a
set $Y \sub \twoom$ is \emph{Marczewski null} (\emph{$Y \in s_0$})
if
for every perfect set~$P \sub \twoom$ there is a perfect
set~$Q \sub P$
with $Q \cap Y = \emptyset$.

Recall the notion of \emph{Sacks forcing} $(\SS, \leq)$,
the set of all perfect subtrees of $\twolom$
with the partial order of inclusion
(i.e.,
$q \leq p$
if
$q \sub p$).
Two elements $p_0, p_1 \in \SS$ are \emph{compatible} if there is an $r \in \SS$ with
$r \leq p_0$ and $r \leq p_1$  (otherwise they are
\emph{incompatible}).
An \emph{antichain} is a collection of pairwise incompatible elements.
Note that an antichain is \emph{maximal} if and only if every
element of $\SS$ is compatible with some element of the antichain.
For $p \in \SS$,
let $[p]$ denote the \emph{body} of~$p$, i.e., the set of branches through the tree $p$.
Note that the set of perfect
sets in
$\twoom$ is exactly
the set of bodies of perfect trees
in $\twolom$
(i.e., $\{[p] : p \in \SS \}$), and that $q \leq p$ if and only if $[q] \sub [p]$.
Therefore,
the following holds:
$Y \in s_0$
if and only if
$\forall p \in \SS \; \exists q \leq p \; [q] \cap Y = \emptyset$;
we will use this ``forcing notation'' from now on.

Since $s_0$ is a $\sigma$-ideal, each countable set is $s_0$-shiftable.
We call the statement that
each $s_0$-shiftable set is countable
``\emph{Borel Conjecture for the Marczewski ideal}'' or
\emph{Marczewski Borel Conjecture} (MBC).
In contrast to the cases above (related to $\M$ and $\N$),
there is no
obvious
way to construct an uncountable $s_0$-shiftable set
from CH (i.e., to show the failure of MBC from CH):
this is due to the fact that (in ZFC) $\cof(s_0) > \cc$
(see~\cite{JudahMillerShelah};
for a quite general investigation of the cofinalities of ideals like this, see~\cite{JoergYuriiIchCofinalities}).
In fact,
the opposite is true:
Theorem~\ref{thm:main_ZFC_result}
(see~Corollary~\ref{cor:CH_implies_MBC}) shows that
CH actually proves MBC
(and so the consistency of MBC is established).

\bigskip

\noindent {\bf Acknowledgments.}
We thank Thilo Weinert for asking whether MBC is consistent, which originated this research,
and for many inspiring discussions with the second author.

Part of the second author's PhD thesis (see~\cite[Chapter 6]{PhDWohofsky}) deals with exploring $s_0$-shiftable sets (mainly under CH),
and contains several of the results of this paper.
He wishes
to thank his advisor Martin Goldstern for many fruitful and inspiring conversations during the course of his PhD.

We also thank
Zbigniew Lipecki
for asking whether our results can be generalized to Polish groups other than $\twoom$.

Part of our work was done in November 2015
at the Isaac Newton Institute for Mathematical Sciences (INI). We thank the INI for its support and hospitality.

Last, but not least, we thank the anonymous referee for his/her careful reading and helpful suggestions.

\subsection{The theorems of the paper}

One of the main results of the paper is the ZFC theorem that there are no
$s_0$-shiftable
sets of reals
of size continuum:

\begin{thm}\label{thm:main_ZFC_result}
(ZFC) Let $X \sub \twoom$ with
$|X| = \cc$.
Then $X$ is not $s_0$-shiftable,
i.e.,
there is a $Y \in s_0$ such that $X + Y = \twoom$.
\end{thm}

Section~\ref{sec:ZFC_result} is devoted to the proof of this theorem.

\begin{cor}\label{cor:CH_implies_MBC}
Assume CH;
then
MBC holds, i.e.,
the collection of $s_0$-shiftable sets is exactly the ideal of countable sets of reals.
In particular, MBC is consistent.
\end{cor}

In the Sacks model, $\cov(s_0) = \cc = \aleph_2$ holds true (see~\cite[Theorem~1.2]{JudahMillerShelah}), i.e., less than
$\cc$
Marczewski null sets do not cover all the reals; it easily follows that any set of size less than $\cc$ is $s_0$-shiftable.
Moreover, $\cov(s_0) = \cc$ is also known to be consistent with
$\cc$
arbitrarily large (see~\cite[Theorem~1 and Proposition~3]{Velickovic_ccc_perfect}).
From this together with Theorem~\ref{thm:main_ZFC_result}
we get that
it is consistent with arbitrarily large continuum
that the collection of $s_0$-shiftable sets
is exactly the ideal of sets of reals of size $< \cc$.

In Section~\ref{sec:arbitrary_Polish_groups}, we demonstrate how to generalize
Theorem~\ref{thm:main_ZFC_result}
to arbitrary Polish groups.

In Section~\ref{sec:MBC_Cohen}, we show that
MBC is consistent with $\lnot$CH, i.e., continuum larger than~$\aleph_1$ (see~Corollary~\ref{cor:MBC_in_the_Cohen_model}):

\begin{thm}\label{thm:MBC_in_Cohen}
MBC holds
in the
Cohen
model
(more precisely: after adding $\kappa$
Cohen reals to a model of GCH,
where
$\kappa \geq \omega_2$
with
$\cf (\kappa) \geq \omega_1$).
\end{thm}

\begin{ques}
Is it consistent that
being
$s_0$-shiftable not only depends on the size?
Is it even consistent that
the collection of $s_0$-shiftable sets does not form an ideal?
\end{ques}

\subsection{Preliminary facts about Marczewski null sets}

Splitting
a
perfect set into ``perfectly many'' (hence continuum many) disjoint perfect sets easily yields

\begin{propos}\label{propos:small_sets_in_s_0}
Let $Y \sub \twoom$ with $|Y| < \cc$. Then $Y \in s_0$.
\end{propos}

Note that the following is a straightforward consequence of Proposition~\ref{propos:small_sets_in_s_0}:
\begin{equation}\label{eqn:equivalent_to_s_0}
Y \in s_0 \iff \forall p \in \SS \qua \exists q \leq p \qua |[q] \cap Y| < \cc;
\end{equation}
in other words,
replacing
``disjoint''
by ``small intersection''
does not change the
notion
of
Marczewski null.

The (proof of the) following theorem
(see also~\cite[Lemma 6.3]{PhDWohofsky}, as well as~\cite[Theorem 5.10]{Miller_special} for a different proof)
is the blueprint for
the Lemmas~\ref{lem:Raach} and~\ref{lem:Cambridge} below:

\begin{thm}\label{thm:thilo}
There is a set $Y \in s_0$ with $|Y| = \cc$.
\end{thm}

\begin{proof}[Proof (Sketch)]
Fix a maximal antichain $\{q_\alpha : \alpha < \cc \} \sub \SS$.
Since $|[q_\alpha] \cap [q_\beta]| \leq \aleph_0$ for any $\alpha \neq \beta$,
we can pick $y_\alpha \in [q_\alpha] \setminus \bigcup_{\beta < \alpha} [q_\beta]$ for each $\alpha < \cc$.
Then
$Y := \{y_\alpha : \alpha < \cc \}$ is of size $\cc$,
and the
maximality of the antichain (i.e., for each $p \in \SS$, there is $q \leq p$ and $\alpha < \cc$ with $q \leq q_\alpha$)
and~\eqref{eqn:equivalent_to_s_0} together yield $Y \in s_0$.
\end{proof}

\section{The ZFC result}\label{sec:ZFC_result}

This section is devoted to the proof of Theorem~\ref{thm:main_ZFC_result}, saying that a set $X$ of size continuum cannot be $s_0$-shiftable.

Our
strategy is as follows:
we first reduce the problem of proving ``$X$ is not $s_0$-shiftable'' to finding a dense and translation-invariant set $D \sub \SS$ with the property that fewer than
$\cc$
trees
from
$D$ do not cover $X$ (see Lemma~\ref{lem:Raach} below);
given a
set $X$ of size $\cc$, we then
show
how to thin it out to a subset $X' \sub X$ of size~$\cc$ in such a way that $X'$ admits a set $D$ with the
aforementioned
properties.

\subsection{Reduction to the set $D$}

\begin{lem}\label{lem:Raach}
Let $X \sub \twoom$, and let $D \sub \SS$ be a
collection of Sacks trees with the following properties:

\begin{enumerate}[(a)]
\item\label{eq:density} $D$ is dense (i.e., for each $p \in \SS$, there is a $q \leq p$ with $q \in D$),

\item\label{eq:translation_invariance} $D$ is translation-invariant (i.e., for each $p \in \SS$ and $t \in \twoom$,
$p \in D$ if and only if
$p + t = \{ \sigma + t \rest |\sigma| : \sigma \in p \} \in D$),

\item\label{eq:no_cover_by_fewer_than_c} fewer than
$\cc$
trees from $D$ do not cover $X$
(i.e.,
if
$\{p_\alpha : \alpha < \mu\} \sub D$
with $\mu < \cc$,
then
$X \nsubseteq \bigcup_{\alpha < \mu} [p_\alpha]$).
\end{enumerate}
Then there is a $Y \in s_0$ such that $X + Y = \twoom$
(in other words, $X$ is not $s_0$-shiftable).
\end{lem}

\begin{proof}
Let $\twoom = \{ z_\alpha : \alpha < \cc \}$ be an enumeration of the reals,
and fix
a maximal
antichain $\{q_\alpha : \alpha < \cc \}$
inside $D$.

By induction on $\alpha < \cc$, we
pick
\begin{equation}\label{eq:picking_x_alpha}
x_\alpha \in X \setminus \bigcup_{\beta < \alpha} (z_\alpha + [q_\beta]);
\end{equation}
this is possible, since
$z_\alpha + q_\beta \in D$
for each $\beta < \alpha$
(see property~\ref{eq:translation_invariance}),
and $X$ cannot be covered by fewer than $\cc$
such sets
(see property~\ref{eq:no_cover_by_fewer_than_c}).

Now let $y_\alpha := x_\alpha + z_\alpha$ for each $\alpha < \cc$,
and define
$Y := \{y_\alpha : \alpha < \cc\}$.
It is clear by construction that
$X + Y = \twoom$
(for each $\alpha < \cc$, we have $x_\alpha \in X$, $y_\alpha \in Y$,
and $z_\alpha = x_\alpha + y_\alpha$).

So it remains to show that
$Y \in s_0$.
Fix $p \in \SS$;
by~\eqref{eqn:equivalent_to_s_0},
it is enough to find a $q \leq p$
such that
$|[q] \cap Y| < \cc$.
Since $D$ is dense (see property~\ref{eq:density} above),
there is $p' \leq p$ with $p' \in D$.
By maximality of the antichain
$\{q_\alpha : \alpha < \cc \}$ within $D$,
we can fix
an $\alpha < \cc$ such that
$q_\alpha$ is compatible with $p'$, i.e.,
we can pick $q \leq p'$ with $q \leq q_\alpha$.
Note that for each $\gamma > \alpha$,
we have $x_\gamma \notin (z_\gamma + [q_\alpha])$ (see~\eqref{eq:picking_x_alpha}),
so $y_\gamma = x_\gamma + z_\gamma \notin [q_\alpha]$,
hence
$y_\gamma \notin [q]$.
Therefore
$[q] \cap Y \sub \{ y_\gamma : \gamma \leq \alpha \}$,
hence
$|[q] \cap Y| < \cc$, as desired.
\end{proof}

\subsection{Transitive versions of being Marczewski null}

We now introduce ``transitive'' versions of being Marczewski null.
Sets
with these
properties will
yield dense sets~$D$ as
needed
in Lemma~\ref{lem:Raach}.

Recall
the
characterization of being
Marczewski null
from~\eqref{eqn:equivalent_to_s_0}.
Requiring that
not only the body of the tree $q$
itself, but also all its translates are ``almost disjoint'' from the set in question, yields transitive versions of the notion of Marczewski null:

\begin{defi}\label{defi:Hejnice}
A set $Y \sub \twoom$ is \emph{\Hejnice{\mu}}
if
\begin{displaymath}
\forall p \in \SS \qua \exists q \leq p \qua
\forall t \in \twoom \qua |([q] + t) \cap Y| < \mu.
\end{displaymath}
Analogously,
a set $Y \sub \twoom$ is \emph{\leqHejnice{\mu}}
if
\begin{displaymath}
\forall p \in \SS \qua \exists q \leq p \qua
\forall t \in \twoom \qua |([q] + t) \cap Y| \leq \mu.
\end{displaymath}
\end{defi}

In~\eqref{eqn:equivalent_to_s_0}, the cardinal $\cc$ can be replaced by any smaller cardinal
without changing the notion;
in Definition~\ref{defi:Hejnice}, however,
the situation is not so clear:

\begin{ques}
To which extent
do
the above notions
depend on $\mu$?
Are
they
strictly stronger than
Marczewski null?
\end{ques}

Requiring disjointness would be definitely too much in the transitive case:
the empty set is the only set which is \leqHejnice{0}.
Typical instances (which we are going to use) are
\Hejnice{\cc} and \leqHejnice{\aleph_0}.
On the other hand,
even finite instances
might
be
worth to consider,
i.e.,
\leqHejnice{\mu} for $\mu \in \omega$.

To finish the proof of Theorem~\ref{thm:main_ZFC_result}, we will proceed as follows: given a set $X$ of size $\cc$, we will show how to find a subset $X' \sub X$
of size $\cc$
which is
\Hejnice{\cc} (see Lemma~\ref{lem:main_lemma_c_regular} below).
In case that $\cc$ is regular, this is sufficient to yield an appropriate dense set $D$ for Lemma~\ref{lem:Raach}, finishing the proof of the theorem.

If $\cc$ is singular, Lemma~\ref{lem:main_lemma_c_regular} might be insufficient, but in this case we are able to strengthen the
(conclusion of the) lemma (see Lemma~\ref{lem:main_lemma_c_singular}) such that we can again obtain a dense $D$
as required in Lemma~\ref{lem:Raach} which finishes the proof of the theorem in ZFC.

\subsection{Luzin sets}

Recall the following classical notions (see~\cite[Definition~8.2.1]{BartoszynskiJudah}): an uncountable set $X \sub \twoom$ is \emph{Luzin}
if
$X \cap M$ is countable for
any meager set $M \in \M$
(such sets exist, e.g., under~CH);
more generally,
a set $X \sub \twoom$
with $|X| = \cc$
is
\emph{generalized Luzin}
if
$|X \cap M| < \cc$ for any $M \in \M$;
furthermore,
a set $X \sub \twoom$
with $|X| = \cc$
is
\emph{generalized Sierpi\'{n}ski}
if
$|X \cap N| < \cc$ for any measure zero set $N \in \N$.

The following lemma says (in ZFC) that there are no
such sets
with respect to the
Marczewski ideal~$s_0$:

\begin{lem}\label{lem:Cambridge}
Let $X \sub \twoom$ with $|X| = \cc$.
Then there exists an $X' \sub X$ with $|X'| = \cc$
such that
$X' \in s_0$.
\end{lem}

\begin{proof}
In case $X \in s_0$,
we can choose $X'$ to be $X$, and we are finished.

So let us assume
that $X \notin s_0$; by~\eqref{eqn:equivalent_to_s_0},
we can fix a $p \in \SS$ satisfying
\begin{equation}\label{eqn:all_below_p_have_size_c_with_X}
\forall q \leq p \qua |[q] \cap X| = \cc.
\end{equation}

Fix
a maximal antichain
$\{q_\alpha : \alpha < \cc \}$
below $p$,
i.e.,
a set
of size $\cc$
such that
\begin{enumerate}[(a)]
 \item\label{it:below_p} $q_\alpha \leq p$ for each $\alpha < \cc$,

 \item\label{it:antichain_below_p}
 $|[q_\alpha] \cap [q_\beta]| \leq \aleph_0$
 for each $\alpha \neq \beta$,

 \item\label{it:maximal_below_p}
 for each $p' \leq p$ there is an $\alpha < \cc$ such that
 $p'$ is compatible with $q_\alpha$.

 \end{enumerate}

We are going to construct a set $X' \in s_0$ of size $\cc$ inside of~$[p] \cap X$
as follows.
By
induction on $\alpha < \cc$,
we
pick
\begin{equation}\label{eq:Cambridge_picking_x_alpha}
x_\alpha \in X \cap [q_\alpha] \setminus \bigcup_{\beta < \alpha} [q_\beta];
\end{equation}
to see that this is possible,
first note
that
$|[q_\alpha] \cap X| = \cc$
(by property~\ref{it:below_p} of the antichain and~\eqref{eqn:all_below_p_have_size_c_with_X});
since
for each $\beta < \alpha$,
we have
$|[q_\beta] \cap [q_\alpha]| \leq \aleph_0$
(by property~\ref{it:antichain_below_p}),
the set
$\bigcup_{\beta < \alpha} [q_\beta]$ cannot cover $X \cap [q_\alpha]$.
Finally,
let
$X' := \{ x_\alpha : \alpha < \cc \}$.
Note that $X' \sub [p]$ (by property~\ref{it:below_p}).

It remains to prove that $X'$ has the desired properties.
It is clear by construction that
$X' \sub X$ and $|X'| = \cc$.
To show that $X' \in s_0$, fix
$p' \in \SS$;
by~\eqref{eqn:equivalent_to_s_0}, it is enough to find
a $q \leq p'$
such that
$|[q] \cap X'| < \cc$.

In case $p'$ is incompatible with $p$ (i.e., $|[p'] \cap [p]| \leq \aleph_0$),
it follows
(since $X' \sub [p]$)
that
$|[p'] \cap X'| \leq \aleph_0 < \cc$,
finishing the proof.

Otherwise (i.e., in case $p'$ is compatible with $p$)
fix $p'' \leq p'$ with $p'' \leq p$.
By maximality below $p$
of our antichain
$\{q_\alpha : \alpha < \cc \}$
(see property~\ref{it:maximal_below_p}),
there is
an $\alpha < \cc$ such that
$q_\alpha$ is compatible with $p''$,
i.e., we can pick $q \leq p''$ with $q \leq q_\alpha$.
As in the proof of Lemma~\ref{lem:Raach}, we can finish the proof as follows:
for each $\gamma > \alpha$,
we have $x_\gamma \notin [q_\alpha] \supseteq [q]$
(see~\eqref{eq:Cambridge_picking_x_alpha});
therefore
$[q] \cap X' \sub \{ x_\gamma : \gamma \leq \alpha \}$, which is of size less than $\cc$.
\end{proof}

\subsection{Skew trees}\label{subsec:skew_trees}

The following notion was defined in~\cite{Blass_skew}:

\begin{defi}\label{defi:skew_in_Sacks}
A Sacks tree
$p \in \SS$
is
\emph{skew}
if for every $n \in \omega$, there is at most one splitting node of level $n$ in
$p$.
\end{defi}

It is easy to see that being skew is dense (see also~\cite[Lemma on page 273]{Blass_skew}) and translation-invariant:

\begin{lem}\label{lem:2_to_the_omega_skew_dense_translation_invariant}
The collection of
skew
trees is (open and) dense in $\SS$,
i.e.,
for every $p \in \SS$, there is a $q \leq p$ such that $q$ is skew (and, whenever $q \in \SS$ is skew, any $r \leq q$ is skew as well).
Moreover,
the
collection
of skew trees is
translation-invariant,
i.e.,
for every skew $q \in \SS$ and $t \in \twoom$,
also $q + t$ is skew.
\end{lem}

We are going to use the following essential property of skew trees:

\begin{lem}\label{lem:2_to_the_omega_skew_lemma}
Let $p \in \SS$ be skew, and let $t \in \twoom$ with $t \neq 0$ (i.e., $t$ has at least one bit with value $1$).
Then
$|[p] \cap ([p] + t) | \leq 2$
(in fact, the intersection has either 2 elements or is empty).
In particular, $p$ and $p + t$ are incompatible.
\end{lem}

\begin{proof}
Assume
that there are
$3$
distinct elements $y_0, y_1, y_2 \in [p] \cap ([p] + t)$, i.e.,
$\{y_0, y_1, y_2, y_0 + t, y_1 + t, y_2 + t \} \sub [p]$.
Due to the nature of
$\twoom$, the reals $y_0, y_1, y_2$ ``split'' at two different levels $n < m$; more precisely: let us assume w.l.o.g.\
that~$n$ is the smallest natural number with $y_0(n) \neq y_1(n)$
and~$m > n$ is the smallest natural number with $y_1(m) \neq y_2(m)$. Furthermore, let $k$ be the smallest natural number with $t(k) = 1$.
Clearly,
either $k < m$ or $k > n$ (or both), and
in either case
this gives rise to two splitting nodes in $p$
at the same level:
if~$k < m$, then
$\{ y_1, y_2 \}, \{ y_1 + t, y_2 + t \}$ give rise to two splitting nodes at
level~$m$; if~$k > n$, then
$\{ y_0, y_0 + t \}, \{ y_1, y_1 + t \}$
give rise to two splitting nodes at level~$k$.
\end{proof}

\subsection{The regular case}

\begin{lem}\label{lem:main_lemma_c_regular}
Let $X \sub \twoom$ with $|X| = \cc$.
Then there exists an $X' \sub X$ with $|X'| = \cc$
such that
$X'$ is \Hejnice{\cc}.
\end{lem}

\begin{proof}
By Lemma~\ref{lem:Cambridge} above, we can assume
w.l.o.g.\ that $X \in s_0$.
We are going to distinguish two cases.

\underline{1st Case}:
$X$ has small intersection with the body of any skew tree:
\begin{equation}\label{eq:1_st_case}
\forall p \in \SS \qua (p \textrm{ skew} \implies |[p] \cap X| < \cc).
\end{equation}
From this it is easy to conclude that
$X$ is \Hejnice{\cc}: given $q \in \SS$, we can find an $r \leq q$
which is skew
(since the skew trees are dense in $\SS$);
consequently, $r + t$ is skew for all $t \in \twoom$
(since being skew is translation-invariant);
therefore
(see~\eqref{eq:1_st_case})
we have $|([r] + t) \cap X| < \cc$
for all $t \in \twoom$,
finishing the proof
that $X$ is \Hejnice{\cc}.

\underline{2nd Case}:
Fix a skew tree $p \in \SS$ such that
$|[p] \cap X| = \cc$.
Define
\begin{displaymath}
X' := [p] \cap X.
\end{displaymath}
Then $X' \sub X$ and
$|X'| = \cc$, so
it only remains to show that
$X'$ is \Hejnice{\cc}
(actually, we are going to show that
$X'$ is even \leqHejnice{\aleph_0}).
First recall
(see Lemma~\ref{lem:2_to_the_omega_skew_lemma})
that
the fact that $p$ is skew implies the following:
\begin{equation}\label{eq:by_skewness}
\forall t \in \twoom \setminus \{ 0 \}:
\quad
p \textrm{ is incompatible with } p + t
\end{equation}
(in other words, $\{p + t : t \in \twoom \} \sub \SS$ is an antichain).
Fix
$q \in \SS$;
we will
find an $r \leq q$ such that
\begin{equation}\label{eq:aleph_0_Hejnice_wish}
\forall t \in \twoom \qua |([r] + t) \cap X'| \leq \aleph_0.
\end{equation}
If $q$ is incompatible with $p + t$ for every $t \in \twoom$, then it is easy to check that $r := q$ satisfies~\eqref{eq:aleph_0_Hejnice_wish}:
for any $t \in \twoom$, we have $|([r] + t) \cap [p]| \leq \aleph_0$, so $X' \sub [p]$ yields~\eqref{eq:aleph_0_Hejnice_wish},
and we are finished.
Otherwise, we can fix a $t_0 \in \twoom$ and
a $q' \leq q$ with
$q' \leq p + t_0$.
Since we have assumed that
$X \in s_0$, we also have
$X' \in s_0$
and hence $X' + t_0 \in s_0$
(since being in $s_0$ is translation-invariant), so we can fix an $r \leq q'$ such that $[r] \cap (X' + t_0) = \emptyset$;
again, it is easy to check that
$r$ satisfies~\eqref{eq:aleph_0_Hejnice_wish}:
if $t = t_0$, we have
$([r] + t) \cap X' = \emptyset$; if $t \neq t_0$, we have
$r + t \leq p + (t_0 + t)$ (with $t_0 + t \neq 0$), so
$r + t$
is
incompatible with $p$ by~\eqref{eq:by_skewness},
i.e.,
$|([r] + t) \cap [p]| \leq \aleph_0$,
hence~\eqref{eq:aleph_0_Hejnice_wish} again holds true by
$X' \sub [p]$,
and the proof is finished.
\end{proof}

Using
the above lemma,
it is easy to finish the proof of the theorem
for the case ``$\cc$~regular'';
the case ``$\cc$~singular''
makes use of another lemma
which
is given below
(Lemma~\ref{lem:main_lemma_c_singular}).

\begin{proof}[Proof of Theorem~\ref{thm:main_ZFC_result}]
Let $X \sub \twoom$ with
$|X| = \cc$.
We have to show that $X$ is not $s_0$-shiftable (i.e.,
there is a $Y \in s_0$ such that $X + Y = \twoom$).

In case $\cc$ is regular,
we apply Lemma~\ref{lem:main_lemma_c_regular}
to obtain a set $X' \sub X$ with $|X'| = \cc$ such that $X'$ is \Hejnice{\cc}.
This gives rise to a set $D$
asked for by Lemma~\ref{lem:Raach}:
we
can fix (see Definition~\ref{defi:Hejnice}) a family
$(q_p : p \in \SS)$ such that for each $p \in \SS$, we have $q_p \leq p$, and
\begin{equation}\label{eq:q_p_property}
\forall t \in \twoom \qua |([q_p] + t) \cap X'| < \cc;
\end{equation}
let $D := \{ q_p + t :\, p \in \SS,\; t \in \twoom \} \sub \SS$;
clearly,
$D$ is dense and translation-invariant
(i.e., $D$ satisfies properties~\ref{eq:density} and~\ref{eq:translation_invariance} in Lemma~\ref{lem:Raach});
moreover,
fewer than
$\cc$
elements from $D$ do not cover $X'$
(i.e., $D$ satisfies property~\ref{eq:no_cover_by_fewer_than_c} in Lemma~\ref{lem:Raach} with respect to~$X'$),
due to the fact that $\cc$ is regular, $X'$ is of size $\cc$, and~\eqref{eq:q_p_property}. Therefore we can apply Lemma~\ref{lem:Raach} to the set $X'$
to derive that $X'$
is not $s_0$-shiftable.
It follows that the same is true for the set $X$, finishing the proof.

In case $\cc$ is singular, we proceed
analogously,
but apply
Lemma~\ref{lem:main_lemma_c_singular}
(instead of
Lemma~\ref{lem:main_lemma_c_regular}) to obtain
a set $X' \sub X$ with $|X'| = \cc$, and a cardinal $\mu < \cc$ such that $X'$ is \leqHejnice{\mu}.
The
family
$(q_p : p \in \SS)$ now satisfies
\begin{equation}\label{eq:q_p_property_singular}
\forall t \in \twoom \qua |([q_p] + t) \cap X'| \leq \mu;
\end{equation}
the corresponding set $D$ again satisfies
properties~\ref{eq:density},
\ref{eq:translation_invariance},
and~\ref{eq:no_cover_by_fewer_than_c}
in Lemma~\ref{lem:Raach}, where~\ref{eq:no_cover_by_fewer_than_c} holds
(even though $\cc$ is singular)
because of the fact that $X'$ is of size $\cc$, and~\eqref{eq:q_p_property_singular}.
The rest of the proof is the same, and so it is finished in ZFC.
\end{proof}

\subsection{The singular case}

The following lemma is more complicated to prove than Lemma~\ref{lem:main_lemma_c_regular}.
Note that, even though its conclusion is stronger,
it is not a
strengthening of the lemma because
there is the additional assumption that $\cc$ is singular.

\begin{lem}\label{lem:main_lemma_c_singular}
Assume $\cc$ is singular. Let $X \sub \twoom$ with $|X| = \cc$.
Then there
exists an
$X' \sub X$ with $|X'| = \cc$ and $\mu < \cc$ such that $X'$ is \leqHejnice{\mu}.
\end{lem}

Note that Lemma~\ref{lem:main_lemma_c_regular} implies that the conclusion of the above lemma also holds if $\cc$ is a successor cardinal.

\begin{ques}
Is this
also
true if $\cc$ is weakly inaccessible?
\end{ques}

\begin{proof}[Proof of Lemma~\ref{lem:main_lemma_c_singular}]
We assume the conclusion is false for
$X$
(in place of $X'$)
and produce $X'$ such that it holds for $X'$.
Let $\muu := \cf (\cc)$.
First we establish:

\begin{claim}
There are an increasing sequence $(\lambda_\alpha)_{\alpha < \muu}$
with $\lambda_\alpha < \cc$ and
$\bigcup_{\alpha < \muu} \lambda_\alpha = \cc$ and a
family $\{p_ \alpha : \alpha < \muu\} \sub \SS$ of
skew
Sacks trees
such that $| [p_\alpha] \cap X | > \lambda_\alpha$,
and
$t + p_\alpha$ and $p_\beta$ are incompatible
for all
$\alpha, \beta < \mu$ and $t \in \twoom$
provided
$\alpha \neq \beta$ or $t \neq 0$.
\end{claim}

\begin{proof}
We distinguish two cases. First assume that there is a skew tree $p \in \SS$ such that for all $p' \leq p$ and all $\lambda < \cc$
there is $t \in \twoom$ such that $| (t + [p'] ) \cap X | > \lambda$. Then split $p$ into $\muu$
trees $p'_\alpha$ ($\alpha < \muu$) with disjoint sets of branches. Next find $t_\alpha$ ($\alpha < \muu$)
such that $| (t_\alpha + [p'_\alpha]) \cap X| > \lambda_\alpha$, where
$(\lambda_\alpha)_{\alpha < \muu}$
is an arbitrary
increasing sequence
of cardinals
with union $\cc$. Let $p_\alpha := t_\alpha + p'_\alpha$. Now fix $t\in\twoom$ and
$\alpha, \beta < \muu$.
If $t + t_\alpha \neq t_\beta$,
then (by skewness of $p$, see Lemma~\ref{lem:2_to_the_omega_skew_lemma})
$| (t+t_\alpha + [p]) \cap (t_\beta + [p])| \leq 2$
and thus $| (t + [p_\alpha]) \cap [p_\beta] | \leq 2$ as required. If $t + t_\alpha = t_\beta$,
then necessarily
$\alpha \neq \beta$ (since otherwise also $t = 0$),
hence
$| (t + [p_\alpha]) \cap [p_\beta] | = 0$ because $[p'_\alpha]$ and $[p'_\beta]$ are disjoint.
This completes the first case.

Thus we may assume that for all skew trees $p$ there are $q \leq p$ and $\lambda < \cc$ such that
$| (t + [q] ) \cap X | \leq \lambda$ for all $t \in \twoom$. On the other hand, since the conclusion of
Lemma~\ref{lem:main_lemma_c_singular}
fails for $X$, we also know that for all $\lambda < \cc$ we can find a
w.l.o.g.\
(since being skew is dense)
skew tree $p$ such that
for all $q \leq p$ there is $t \in\twoom$ with $| (t + [q] ) \cap X | > \lambda$. Taking these two
assumptions together,
it is straightforward to construct a
strictly increasing sequence
of cardinals
$(\lambda_\alpha)_{\alpha < \muu}$
with union $\cc$ and a sequence of skew trees
$(p_\alpha)_{\alpha < \muu}$
such that for all $q \leq p_\alpha$ there is $t \in \twoom$ such that $| (t + [q]) \cap X | > \lambda_\alpha$
and $| (t + [p_\alpha] ) \cap X | \leq \lambda_\beta$
for all $t \in \twoom$ and all $\beta > \alpha$.

To see that the $p_\alpha$ are as required, fix $t \in\twoom$ and $\alpha, \beta < \muu$;
we have to show that $t + p_\alpha$ is incompatible with $p_\beta$.
In case $\alpha = \beta$,
we have
$t \neq 0$, and the skewness of $p_\alpha$ yields the incompatibility.
So,
w.l.o.g.,
$\alpha < \beta$. Assume $t + p_\alpha$ and $p_\beta$ are compatible with
common extension $q$. Since $q \leq t + p_\alpha$, we know that $ | (t' + [q] ) \cap X | \leq \lambda_\beta$
for all $t' \in \twoom$. On the other hand, $q \leq p_\beta$ implies that there is $t' \in \twoom$
such that $| (t' + [q] ) \cap X | > \lambda_\beta$, a contradiction.
\end{proof}

We now complete the proof of
Lemma~\ref{lem:main_lemma_c_singular}
using the claim.
As in the regular case, we can again assume
w.l.o.g.\
(see Lemma~\ref{lem:Cambridge})
that $X \in s_0$.
Let $X ' := \bigcup_{\alpha < \muu} ([p_\alpha] \cap X)$ where the $p_\alpha$ are as in the claim.
Clearly, $X' \sub X$, and $|X'| = \cc$ by the claim.
So it just remains to prove that $X'$ is \leqHejnice{\mu}.

Fix
$p \in \SS$; we have to find $q \leq p$ such that
$|([q] + t) \cap X'| \leq \mu$ holds for each $t \in \twoom$.
If for all $\alpha < \muu$
and all $t \in \twoom$,
$t + p$ is incompatible with $p_\alpha$, then $q = p$ clearly satisfies the conclusion.
Hence assume that $t + p$ is compatible with $p_\alpha$ for some $\alpha < \muu$ and $t \in \twoom$.
Let $p' \leq p$ be such that $t + p' \leq p_\alpha$.
It follows from the
claim
that $t + p'$ is incompatible with
all $t' + p_\beta$
whenever
$\alpha \neq \beta$ or $t' \neq t$.
Now let $q \leq p'$ such that $(t + [q] ) \cap X' = \emptyset$. Such $q$ exists because $X$ (and hence $X'$) is in~$s_0$.
We shall see that $q$ is as required.

Let $t' \in \twoom$ be arbitrary. If $t' = t$, we are done.
If $t' \neq t$, then $t' + t \neq 0$ and therefore $t + q$ is incompatible with all
$t' + t + p_\beta$,
that is, $t' + q$ is incompatible with all
$p_\beta$,
and $| (t' + [q] ) \cap X' | \leq
| (t' + [q] ) \cap (\bigcup_{\beta < \muu}  [p_\beta] )| \leq \muu$, as required.
This completes the proof of the
lemma.
\end{proof}

\section{Skew perfect sets in arbitrary Polish groups}\label{sec:arbitrary_Polish_groups}

In this section, we are going to explain how to generalize Theorem~\ref{thm:main_ZFC_result} to arbitrary Polish groups.

Let $(G,+)$ be a Polish group.
While there might be no generalization of Lebesgue measure,
most other concepts mentioned in the introduction (see Section~\ref{sec:Introduction}) can be canonically
interpreted in $(G,+)$:
first of all,
for each $X, Y \sub G$ and $t \in G$, $X + Y$ and $X + t$ are defined, and so is $\I$-shiftability\footnote{Note that there are two versions of $\I$-shiftability in case of a non-abelian group,
such as two notions of $s_0$, of MBC, etc. However, the results presented here hold true for both versions, i.e., there is no $s_0$-shiftable set of size $\cc$ of either type; we do not claim though that the two notions necessarily coincide.}
for any $\I \sub \P(G)$; moreover,
since $G$ is a topological group, we have the notions of closed set, isolated point,
meagerness,
etc.; in particular,
a set $P \sub G$ is perfect if it is closed and has no isolated points, and a
set $Y \sub G$ is \emph{Marczewski null} (\emph{$Y \in s_0$})
if
for every perfect set~$P \sub G$ there is a perfect
set~$Q \sub P$
with $Q \cap Y = \emptyset$; consequently, it is natural to define
\mbox{``\emph{$s_0$-shiftable}''}
and
``\emph{Marczewski Borel Conjecture} (MBC)'' in any Polish (or even any topological) group.

Theorem~\ref{thm:main_ZFC_result}
turns out to hold true in any Polish group.
It is quite straightforward to check that the proof can be done
in a way
completely analogous
to the one for $(\twoom,+)$ given in Section~\ref{sec:ZFC_result},
using
analoguous versions of
Lemma~\ref{lem:Raach},
Definition~\ref{defi:Hejnice},
Lemma~\ref{lem:Cambridge},
as well as
Lemma~\ref{lem:main_lemma_c_regular} and
Lemma~\ref{lem:main_lemma_c_singular}.
The only essential modification
concerns
the material presented in
Subsection~\ref{subsec:skew_trees}
involving
the notion of skewness of a tree (whose combinatorial definition is
very
specific to $\twoom$);
so the scope of this section
is to provide a generalized
definition of skewness
(Definition~\ref{defi:goettingen_skew}, replacing Definition~\ref{defi:skew_in_Sacks})
and to prove that it enjoys the desired properties:
the property originally given by
Lemma~\ref{lem:2_to_the_omega_skew_lemma} is now
given by
Lemma~\ref{lem:skew_only_two}, whereas translation-invariance and density of skewness (see Lemma~\ref{lem:2_to_the_omega_skew_dense_translation_invariant})
is now
provided by Lemma~\ref{lem:skew_translation_invariance}
and, most importantly, Lemma~\ref{lem:skew_density}.

Note that
Definition~\ref{defi:goettingen_skew}
as well as
the two lemmas before the main Lemma~\ref{lem:skew_density}
work for
any group
$(G,+)$.

\begin{defi}\label{defi:goettingen_skew}
A set $Z \sub G$ is \emph{skew} if for all $x, y, v, w \in Z$ we have
\[
x \neq y \;\land\; v \neq w \;\land\;
\{x,y\} \neq \{v,w\}
\implies
x - y \neq v - w.
\]
\end{defi}

\begin{lem}\label{lem:skew_only_two}
Assume $Z \sub G$ is skew and $t \in G$ with $t \neq 0$. Then
$|Z \cap (Z + t)| \leq 2$.
\end{lem}

\begin{proof}
Assume towards a contradiction that $\{a, b, c\} \sub Z$ with $|\{a, b, c\}| = 3$ and
$\{a - t, b - t, c - t\} \sub Z$.

Since $a \neq c$, either
$b - t \neq a$
or
$b - t \neq c$;
say, w.l.o.g.,
$b - t \neq a$ holds.
Let
$x := a$, $y := b$, $v := a - t$, and $w := b - t$.
Then $x, y, v, w \in Z$, $x \neq y$, $v \neq w$, and $\{x,y\} \neq \{v,w\}$, but $x - y = v - w$, a contradiction (see Definition~\ref{defi:goettingen_skew}).
\end{proof}

\begin{lem}\label{lem:skew_translation_invariance}
Being skew is translation-invariant, i.e.,
whenever $Z \sub G$ is skew
and $t \in G$, then $Z + t$
is skew as well.
\end{lem}

Finally, being skew is ``dense'' within the collection of all perfect subsets of~$G$:

\begin{lem}\label{lem:skew_density}
Let $(G,+)$ be a Polish group, and let $P \sub G$ be a perfect set. Then there is a perfect set $Q \sub P$ such that $Q$ is skew.
\end{lem}

\begin{proof}
We start with a definition, which
makes
sense in every group,
and
two lemmas,
which work in every topological group.

\begin{defi}\label{defi:skew_like}
Let
$(A_n : n \in m)$ be a sequence of length $m \in \omega$ with $A_n \sub G$
for each $n \in m$.

We
say that $(A_n : n \in m)$
is
\emph{skew-like}
if for each
$(i,j,k,l) \in m^4$
satisfying $i \neq j$, $k \neq l$, $\{i,j\} \neq \{k,l\}$
the following holds:
\begin{equation}\label{eq:skew_like}
(x,y,v,w) \in A_i \times A_j \times A_k \times A_l \implies x-y \neq v-w.
\end{equation}

For a fixed quadruple
$(i,j,k,l) \in m^4$
satisfying $i \neq j$, $k \neq l$, $\{i,j\} \neq \{k,l\}$,
we
say that $(A_n : n \in m)$ is
\emph{skew-like with respect to $(i,j,k,l)$}
if~\eqref{eq:skew_like}
holds
for this quadruple.
\end{defi}

Note that being skew-like is preserved when
sets are replaced by subsets:
whenever $(A_n : n \in m)$ is
skew-like (with respect to $(i,j,k,l)$),
then so is any $(A_n' : n \in m)$
satisfying $A_n' \sub A_n$ for each $n \in m$.

To simplify notation, we say that $(z_n : n \in m)$ (with $z_n \in G$) is skew-like
if $(\{z_n\} : n \in m)$ is skew-like according to the above definition (analogous for skew-like with respect to $(i,j,k,l)$).

\begin{lem}\label{lem:find_skew_like_neighborhoods}
Assume $(i,j,k,l) \in m^4$ is given, and
$(z_n : n \in m)$
is skew-like with respect to $(i,j,k,l)$.

Then there are open neighborhoods $(U_n : n \in m)$ of $(z_n : n \in m)$ such that $(U_n : n \in m)$ is skew-like with respect to $(i,j,k,l)$.
\end{lem}

\begin{proof}
First note that the assumption (i.e., $(z_n : n \in m)$
is skew-like with respect to $(i,j,k,l)$)
is exactly the statement $z_i - z_j \neq z_k - z_l$; in other words,
the quadruple $(z_i,z_j,z_k,z_l)$ belongs to the set
$H := \{(x,y,v,w) : x-y \neq v-w \} \sub G^4$.
The mapping
$\varphi: G^4 \rightarrow G$
with
$(x,y,v,w) \mapsto x-y+w-v$
is continuous, and $G \setminus \{0\}$ is open, hence $H = \varphi^{-1}(G \setminus \{0\})$ is open as well.

Consequently,
we can fix open neighborhoods $(I,J,K,L)$ of $(z_i,z_j,z_k,z_l)$
with
$I \times J \times K \times L \sub H$,
so it is clear that we can find
open
neighborhoods $(U_n : n \in m)$ of $(z_n : n \in m)$
such that
$U_i \times U_j \times U_k \times U_l \sub I \times J \times K \times L \sub H$.

By definition of $H$,
$(U_n : n \in m)$ is skew-like with respect to $(i,j,k,l)$.
\end{proof}

\begin{lem}\label{lem:main_skew_like}
Let $P \sub G$ be a set without isolated points.

Given $(z_n : n \in m)$ and $(U_n : n \in m)$
satisfying (for each $n \in m$)
\begin{enumerate}[(A)]
\item\label{eq:item_a} $z_n \in P$, and
\item\label{eq:item_b} the set $U_n$ is an open neighborhood of $z_n$,
\end{enumerate}
there are
$(z_n' : n \in m)$ and $(U_n' : n \in m)$ satisfying
(for each $n \in m$)
\begin{enumerate}[(a)]
\item\label{eq:item_A} $z_n' \in P$,
\item\label{eq:item_B} $U_n'$ is an open neighborhood of $z_n'$, and
\item\label{eq:item_C} $U_n' \sub {U_n}$,
\end{enumerate}
such that
$(U_n' : n \in m)$ is skew-like.\footnote{Formally, the $z_n$ and $z_n'$ are not needed in the
statement of this lemma, but they play an important role in the proof and in the Cantor scheme of 
the proof of Lemma~\ref{lem:skew_density}.}
\end{lem}

\begin{proof}
First
note
that it is enough
(recursively deal with the finitely many quadruples)
to prove the assertion of the lemma relativized to
``with respect to $(i,j,k,l)$'', i.e., it suffices to show that for each fixed quadruple
$(i,j,k,l) \in m^4$
satisfying $i \neq j$, $k \neq l$, $\{i,j\} \neq \{k,l\}$,
the following holds:
given
$(z_n : n \in m)$ and $(U_n : n \in m)$
satisfying~\ref{eq:item_a} and~\ref{eq:item_b},
there are
$(z_n' : n \in m)$ and $(U_n' : n \in m)$
satisfying~\ref{eq:item_A}, \ref{eq:item_B}, and~\ref{eq:item_C},
such that
$(U_n' : n \in m)$ is
skew-like
with respect to $(i,j,k,l)$.

So for the rest of the proof, let us fix an
$(i,j,k,l)$ satisfying
$i \neq j$, $k \neq l$, $\{i,j\} \neq \{k,l\}$.

Assume
$(z_n : n \in m)$ and $(U_n : n \in m)$
satisfying~\ref{eq:item_a} and~\ref{eq:item_b} are given.
We are going to prove that
there is $(z_n' : n \in m)$ with $z_n' \in P \cap U_n$ (for each $n \in m$)
such that
$(z_n' : n \in m)$ is skew-like with respect to $(i,j,k,l)$,
i.e.,
\begin{equation}\label{eq:singleton_skew_like}
z_i' - z_j' \neq z_k' - z_l'.
\end{equation}
By the lemma above,
this
is enough to finish the proof:
once we have
$(z_n' : n \in m)$,
Lemma~\ref{lem:find_skew_like_neighborhoods} yields $(U_n' : n \in m)$, skew-like with respect to $(i,j,k,l)$;
w.l.o.g., we can assume
$U_n' \sub U_n$ for each $n$, so
\ref{eq:item_A}, \ref{eq:item_B}, and~\ref{eq:item_C} are fulfilled, and being skew-like remains true.

Now observe the following: since
$(i,j,k,l)$ satisfies
$i \neq j$, $k \neq l$, $\{i,j\} \neq \{k,l\}$,
the set $\{i,j,k,l\}$
has size at least $3$, so
there is
an
element
(in fact, at least two)
in
$\{i,j,k,l\}$
which appears at exactly one position in the quadruple $(i,j,k,l)$.
From now on,
distinguish
this position; for the sake of notational simplicity only, assume that it is the first position, i.e., $i \notin \{j,k,l\}$.

We have to find $(z_n' : n \in m)$ with $z_n' \in P \cap U_n$ (for each $n \in m$)
such that~\eqref{eq:singleton_skew_like} holds true.
To that end, let $z_n' := z_n$ for each $n \in m \setminus \{i\}$.
To determine $z_i'$,
consider the equation $x - z_j' = z_k' - z_l'$. Since the free variable $x$ appears
only
once,
it is possible to solve the equation to uniquely determine $x \in G$. Since $P \sub G$
has no isolated point,
$P \cap U_i$ contains elements distinct from $x$. Let
$z_i'$
be any such element; then~\eqref{eq:singleton_skew_like} holds true, and the proof of the lemma is finished.
\end{proof}

We are now ready to finish the proof of Lemma~\ref{lem:skew_density}.
Since $G$ is a Polish group, we can
fix a
complete compatible metric $d$.
(From now on, all balls are
understood
with respect to $d$.)
Given a perfect set $P \sub G$,
we
are going to construct a Cantor scheme
(see, e.g., \cite[Definition~6.1 and Theorem 6.2]{Kechris})
in order to define a set~$Q$,
a homeomorphic copy of $\twoom$
within the perfect set~$P$,
which is in addition skew.

More precisely, we
define
$(z_s : s \in \twolom)$ and $(B_s : s \in \twolom)$
satisfying the following properties:
for
each $s \in \twolom$,
\begin{enumerate}
\item\label{first_of_five} $z_s \in P$,
\item $B_s$ is a closed ball with center $z_s$,
\item the diameter of $B_s$ is less than $2^{-|s|}$,
\item $B_{s^ \frown 0} \sub B_s$ and $B_{s^ \frown 1} \sub B_s$,
\item\label{last_of_five} $B_{s^ \frown 0} \cap B_{s^ \frown 1} = \emptyset$;
\end{enumerate}
moreover, for each $n \in \omega$,
\begin{enumerate}
\item[(6)] $(B_s : s \in 2^n)$ is skew-like.
\end{enumerate}

It is a standard
straightforward
construction to obtain $(z_s : s \in \twolom)$ and $(B_s : s \in \twolom)$ satisfying properties~(\ref{first_of_five})--(\ref{last_of_five}) above.
In order to ensure property~(6), apply Lemma~\ref{lem:main_skew_like} at each level of the construction: just note that the number of neighborhoods considered at level $n$ is finite (in fact, $2^n$), so it is possible to use the lemma to get $z_s \in P$ and balls $B_s$ such that property~(6) holds true as well.

Now we can define $Q$ as the set of limits of the $z_s$'s along branches of the binary tree:
for each $\rrr \in \twoom$, the sequence $(z_{\rrr \restrict n} : n \in \omega)$ is a Cauchy sequence (with respect to the metric $d$), so (since $G$ is complete with respect to $d$) the sequence $(z_{\rrr \restrict n} : n \in \omega)$ converges to a point $z_\rrr \in G$; finally, let $Q := \{ z_\rrr : \rrr \in \twoom \}$.

It is easy to check that $Q$ is perfect, and that $Q \sub P$ (due to the fact that each~$z_s$ is in $P$,
and $P$ is closed).

So it remains to show that $Q$ is skew.
Let $x,y,v,w \in Q$
satisfy
\begin{equation}\label{eq:skew_hypothesis}
x \neq y, \; v \neq w, \; \{ x,y \} \neq \{ v,w \};
\end{equation}
we will show that $x - y \neq v - w$.
Fix $\rrr_x, \rrr_y, \rrr_v, \rrr_w \in \twoom$
such that $z_{\rrr_x} = x$, $z_{\rrr_y} = y$, $z_{\rrr_v} = v$, $z_{\rrr_w} = w$.
By~\eqref{eq:skew_hypothesis}, we have
$\rrr_x \neq \rrr_y, \rrr_v \neq \rrr_w, \{ \rrr_x,\rrr_y \} \neq \{ \rrr_v,\rrr_w \}$.
Therefore we can fix an $n \in \omega$ large enough so that the latter is reflected down to level~$n$, i.e.,
\begin{equation}\label{eq:skew_like_hypothesis}
\rrr_x \restrict n \neq \rrr_y \restrict n, \; \rrr_v \restrict n \neq \rrr_w \restrict n, \;
\{ \rrr_x \restrict n, \rrr_y \restrict n \} \neq \{ \rrr_v \restrict n, \rrr_w \restrict n \}.
\end{equation}
By construction (see property~(6) above), $(B_s : s \in 2^n)$ is skew-like
(see Definition~\ref{defi:skew_like}), and
$(\rrr_x \restrict n, \rrr_y \restrict n, \rrr_v \restrict n, \rrr_w \restrict n) \in (2^n)^4$
satisfies~\eqref{eq:skew_like_hypothesis}.
Now observe that $x \in B_{{\rrr_x} \restrict n}$
(since $B_{{\rrr_x} \restrict n}$ is closed,
and for each $k \geq n$,
$z_{{\rrr_x} \restrict k} \in B_{{\rrr_x} \restrict k} \sub B_{{\rrr_x} \restrict n}$,
hence
$x = z_{\rrr_x} \in B_{{\rrr_x} \restrict n}$),
$y \in B_{{\rrr_y} \restrict n}$,
$v \in B_{{\rrr_v} \restrict n}$, and
$w \in B_{{\rrr_w} \restrict n}$,
so $x - y \neq v - w$, and the proof of the lemma is finished.
\end{proof}

\section{Marczewski Borel Conjecture in the Cohen model}\label{sec:MBC_Cohen}

This section is devoted to the proof of Theorem~\ref{thm:MBC_in_Cohen}:
in the Cohen model, MBC holds (i.e., there is no uncountable $s_0$-shiftable set in the Cohen model;
see Corollary~\ref{cor:MBC_in_the_Cohen_model}).

Assume $\bb = \aleph_1$ and let $\{ g_\alpha : \alpha < \omega_1 \}$ be a well-ordered unbounded family of functions, that is,
$\alpha < \beta $ implies $g_\alpha \leq^* g_\beta$ and for all $f \in \omom$ there is $\alpha < \omega_1$ with
$g_\alpha \not\leq^* f$. Given a Sacks tree $T \in \SS$, define the following two functions describing the ``speed" with
which splitting occurs in $T$:
\[ h_T (n) := \min \{ k : \mbox{ some node in } T \mbox{ at level } k \mbox{ has } n \mbox{ splitting predecessors} \} \]
and
\[ f_T (n) := \min \{ k : \mbox{ every node in } T \mbox{ at level } k \mbox{ has } 2n \mbox{ splitting predecessors} \}. \]
Clearly $h_T (n) \leq f_T (n)$ for all $n$,
and, for $x \in \twoom$, $h_T = h_{x+T}$ and $f_T = f_{x+T}$.

Assume $S,T \in \SS$. Say that $S$ is {\em somewhere dense in} $T$ if there is $s \in T$ with $T_s \leq S$
(where $T_s = \{ t \in T : t \subseteq s \lor s \subseteq t \}$). If there is
no such $s$, $S$ is {\em nowhere dense in} $T$. The latter is clearly equivalent to saying that $[S] \cap [T]$ is
nowhere dense in the relative topology of $[T]$, considered as a subspace of the Cantor space $\twoom$.

\begin{obs}  \label{splitting-obs}
Let $S,T \in \SS$, and assume that $h_S \not\leq^* f_T$.
Then $S$ is nowhere dense in $T$.
\end{obs}

\begin{proof}
Let~$\predec{s}{T}$ denote the number of
splitting predecessors of~$s$ within~$T$.

Fix $s \in T$; we have to prove that $T_s \leq S$ fails, i.e., we have to find an extension of~$s$ which is in~$T$ but not in~$S$.
By assumption, we can fix $n \in \omega$ such that
$n \geq \predec{s}{T} - \predec{s}{S}$,
$n \geq |s|$, and
$h_S(n) > f_T(n)$.
By definition of~$f_T$, we have
$f_T(n) \geq 2n$, hence $f_T(n) \geq |s|$.
So we can fix $s' \in T$ with $s' \supseteq s$ and $|s'| = f_T(n)$;
in case $s' \notin S$, the proof is finished, so let us assume that $s' \in S$.
Observe that $\predec{s'}{T} \geq 2n$
(by definition of~$f_T$);
since $h_S(n) > f_T(n)$,
we have
$|s'| < h_S(n)$ and hence
$\predec{s'}{S} < n$
(by definition of~$h_S$).
Note that
$\predec{s'}{T} - \predec{s'}{S} > n \geq \predec{s}{T} - \predec{s}{S}$,
so
we can find $t \in T \cap S$
with
$s \subseteq t \subsetneq s'$
such that
$t$ is a splitting node of~$T$, but not a splitting node of~$S$;
consequently,
there is $i \in 2$ such that $s \subseteq t^\frown i \in T \setminus S$, as desired.
\end{proof}

Let $G \leq \twoom$ be a group. Say that $\bar\T = \la \T_\alpha : \alpha < \omega_1 \ra$ is a {\em $G$-matrix}
of Sacks trees if for each $\alpha < \omega_1$,
\begin{romanenumerate}
\item $\T : = \bigcup_{\alpha < \omega_1} \T_\alpha \sub \SS$ consists of skew Sacks trees,
\item $g_\alpha \leq^* h_T$ for all $T \in \T_\alpha$,
\item for all $S \neq T$ in $\T_\alpha$ and all $x\in G$, the trees $x + S$ and $T$ are incompatible
   (equivalently, $| (x + [S] ) \cap [T] | \leq \aleph_0$),
\item for all $T \in \SS$, the set $\{ (x,S) \in G \times \T :  x + S$ is somewhere dense in $T\}$ is at most countable.
\end{romanenumerate}
$\bar\T$ is a {\em dense $G$-matrix} if additionally
\begin{romanenumerate} \setcounter{enuroman}{4}
\item $\T$ is dense in $\SS$.
\end{romanenumerate}
The last property in the definition of $G$-matrix is actually redundant. We kept it in the list because it plays a crucial role
in the proof.

\begin{obs}
Property (iv) follows from properties (i) through (iii).
\end{obs}

\begin{proof} Fix $T \in \SS$. There is $\alpha < \omega_1$ such that $g_\alpha \not\leq^* f_T$. By the previous observation and (ii),
we know that for all $\beta \geq \alpha$, all $S \in \T_\beta$ and all $x \in G$, $x + S$ is nowhere dense in $T$.
So let $\beta < \alpha$ and fix $s \in \twolom$. If $T_s \leq x + S$ and $T_s \leq x' + S'$ for $x,x' \in G$ and
$S , S' \in \T_\beta$, then $S$ and $x+x' + S'$ are compatible and $S = S'$ follows by (iii). Skewness (i), then,
implies that $x=x'$
(see Lemma~\ref{lem:2_to_the_omega_skew_lemma}). This establishes~(iv).
\end{proof}

\begin{lem}   \label{firstextension}
{\rm (First extension lemma)} Assume $G \in V$ and $\bar \T \in V$ is a $G$-matrix. Let $S \in \SS \cap V$. Then,
in the Cohen extension $V[c]$, there are $T \leq S$ and $\alpha < \omega_1$ such that $\bar \T'$ given
by $\T_\alpha'= \T_\alpha \cup \{ T \}$ and $\T_\beta ' = \T_\beta$ for $\beta \neq \alpha$ is still a $G$-matrix.
\end{lem}

\begin{proof}
Again choose $\alpha < \omega_1$ such that $g_\alpha \not\leq^* f_S$. By Observation~\ref{splitting-obs} and (ii)
we know that for all $\beta \geq \alpha$, all $U \in \T_\beta$ and all $x \in G$, $x + U$ is nowhere dense in $S$.
In $V[c]$, let $T \leq S$ be a perfect tree of Cohen reals (in the topology of $[S]$). That is, $T$ is obtained by forcing
with finite subtrees of $S$ ordered by end-extension; this is equivalent to Cohen forcing. By pruning $T$ further,
if necessary, we may assume that $T$ is skew and that $g_\alpha \leq^* h_T$. By Cohen-genericity we have that for all $x \in G$
and all $U \in \T$ such that $x+U$ is nowhere dense in $S$, $(x + [U]) \cap [T] = \emptyset$.
In particular, $x + U$ and $T$ are incompatible. Thus (i) through (iii) for $\bar\T'$ are immediate.
\end{proof}

\begin{lem}  \label{secondextension1}
{\rm (Second extension lemma, successor step)} Assume $X \sub \twoom$, $X \in V$, has size at least $\aleph_1$, and let $G = \la X \ra$
be the group generated by $X$. Also assume $\bar \T \in V$ is a $G$-matrix. Let $\dot z$ be a $\CC$-name
for a new real (that is, the trivial condition forces $\dot z \notin V$). Then there is $x = x_{\dot z} \in X$ such that
the trivial condition forces $\dot z \notin x + [T]$ for all $T \in\T$.
\end{lem}

\begin{proof}
For all conditions $p \in \CC$, let $T_{\dot z, p}$ be the tree of possibilities for $\dot z$ below $p$;
that is,
$s \in T_{\dot z, p}$ if there is $q \leq p$ such that $q \forces s \sub \dot z$. Since $\dot z$ is a name for
a new real, $T_{\dot z, p}$ is a perfect tree for every $p \in \CC$.

For each $p\in \CC$, by (iv), $ \{ (x,S) \in G \times \T : x + S$ is somewhere dense in $T_{\dot z, p} \}$ is at most
countable. Hence we may choose $x \in X$ such that $x + S$ is nowhere dense in $T_{\dot z,p}$ for all $S \in \T$ and
all $p \in \CC$. This implies that $\dot z \notin x + [S]$ for all $S \in\T$ is forced by the empty condition.
(To see this, fix $p \in \CC$. Since $x+S$ is nowhere dense in $T_{\dot z,p}$, there is $s \in T_{\dot z,p}$ with
$s \notin x + S$. So we can find $q \leq p$ such that $q \forces s \sub \dot z$, i.e., $q \forces \dot z \notin x + [S]$.)
\end{proof}

\begin{lem}  \label{secondextension2}
{\rm (Second extension lemma, limit step)} Assume $X \sub \twoom$, $X \in V$, has size at least $\aleph_1$, and let $G = \la X \ra$.
Put $V_0 =V$, let $V_n$, $n \in \omega$, be an increasing chain of intermediate ccc extensions, and let $V_\omega$
be the extension obtained by the direct limit. Assume
$\la \bar \T_n : n \in \omega \ra$ is
a chain
of $G$-matrices such that $\bar \T_n \in V_n$ for each $n \in \omega$ and $\la \T_n : n \in \omega \ra$ is increasing.
Let $z \in V_\omega$ be a new real (i.e., $z \notin V_n$ for any $n$). Then there is $x = x_{z} \in X$ such that
$z \notin x + [T]$ for all $T \in\bigcup_n\T_n$.
\end{lem}

\begin{proof}
In $V_n$, let $T_{\dot z,n}$ be the tree of possibilities for $\dot z$: $s \in T_{\dot z,n}$ if there is $p$ in the quotient forcing
$\PP_{[n,\omega)}$ leading from $V_n$ to $V_\omega$ such that $p \forces s \sub \dot z$. Again $T_{\dot z,n}$ must
be a perfect tree.

By (iv), $\{ (x,S) \in G \times \T_n : x+S$ is somewhere dense in $T_{\dot z,n} \}$ is at most
countable. That is, $X_n := \{ x \in G : $ there is $S\in \T_n$ such that $x + S$ is somewhere dense
in $T_{\dot z,n}\}$ is at most countable.

In the ground model, $Y_n : = \{ x \in G : $ some $p \in \PP_n$ forces $x \in \dot X_n \}$ is at most countable
by the ccc of the forcing $\PP_n$ leading to the generic extension $V_n$. So we can find $x \in X$ with $x\notin \bigcup_n Y_n$.
We claim that $z \notin x + [S]$ for all $S \in \bigcup_n \T_n$.

To see this, fix $n$ and $S \in \T_n$ in $V_n$. Also let
$p \in \PP_{[n,\omega)}$. Then $p \in \PP_{[n,m)}$ for some $m \geq n$. Step into $V_m$ with $p$
belonging to the generic. Since $x\notin X_m$ and $S \in \T_m$,
$x + S$ is nowhere dense in $T_{\dot z,m}$. Find a condition $q$ in $\PP_{[m,\omega)}$
and $s \in T_{\dot z,m} \sem x + S$ such that $q \forces s \sub \dot z$. Then $q\forces \dot z \notin x + [S]$.
\end{proof}

\begin{cor}\label{cor:MBC_in_the_Cohen_model}
Assume $V \models GCH$. Add
$\kappa$
Cohen reals, $\kappa \geq \omega_2$, $\cf (\kappa) \geq \omega_1$.
Then the following holds in the generic extension $V_\kappa$.
Assume $X \sub \twoom$, $|X| = \aleph_1$, and let $G = \la X \ra$. There is a
dense $G$-matrix $\bar\T$ such that additionally for all $z \in \twoom$ there is $x_z \in X$ such that
$z \notin x_z + [T]$ for all $T \in \T$.
In particular, if $Y = \twoom \sem \bigcup \{[T]:T\in\T\}$, then
$Y \in s_0$ and $X + Y =\twoom$.
\end{cor}

\begin{proof}
Let $V_0 = V$ and denote by $V_\gamma$, $\gamma \leq \kappa$, the $\CC_\gamma$-generic extension.
W.l.o.g.\ $X \in V_0$.
Let $\{ S_\gamma : \gamma < \kappa \}$ list all Sacks trees in $V_\kappa$
such that $S_\gamma \in V_\gamma$ for all $\gamma < \kappa$.

First, using the first extension lemma (Lemma~\ref{firstextension}), we build $G$-matrices $\bar \T_\gamma^\star =
\la \T_{\alpha,\gamma}^\star : \alpha < \omega_1 \ra$ for $\gamma \leq \kappa$ with $\bar \T_\gamma^\star \in V_\gamma$, such
that
\begin{itemize}
\item $\T_{\alpha,\gamma}^\star \sub \T_{\alpha,\delta}^\star$ for $\gamma < \delta \leq\kappa$,
\item for some $T_\gamma^* \leq S_\gamma$ and some $\alpha = \alpha_\gamma < \omega_1$, $\T_{\alpha,\gamma + 1}^\star =
   \T_{\alpha,\gamma}^\star \cup \{ T_\gamma^* \}$ and $\T_{\beta,\gamma + 1}^\star =
   \T_{\beta,\gamma}^\star$ for $\beta \neq \alpha$,
\item $\T_{\alpha,\gamma}^\star = \bigcup_{\delta<\gamma} \T_{\alpha,\delta}^\star$ for limit $\gamma$.
\end{itemize}
Then clearly $\bar \T_\kappa^\star$ is a dense $G$-matrix.

Next, let $z \in \twoom$ (in $V_\kappa$). There is a minimal $\gamma < \kappa$ such that $z \in V_\gamma$,
and
this~$\gamma$ satisfies
$\cf (\gamma) \leq \omega$. If $\gamma = \delta +1$ is successor, apply Lemma~\ref{secondextension1}
to obtain $x_z \in X$ such that $z \notin x_z + [T]$ for all $T \in \T_\delta^\star$. If $\gamma$ is a limit
of countable cofinality, apply Lemma~\ref{secondextension2} to obtain $x_z \in X$ such that
$z \notin x_z + [T]$ for all $T \in \T_\gamma^\star = \bigcup_{\delta<\gamma}\T_\delta^\star$.

Finally, build $G$-matrices $\bar \T_\gamma = \la \T_{\alpha,\gamma} : \alpha < \omega_1 \ra$
for $\gamma \leq \kappa$ with $\bar \T_\gamma \in V_\gamma$, such that
\begin{itemize}
\item $\T_{\alpha,\gamma} \sub \T_{\alpha,\delta}$ for $\gamma<\delta\leq\kappa$,
\item for some $T_\gamma \leq T_\gamma^*$, $\T_{\alpha_\gamma,\gamma+2} = \T_{\alpha_\gamma,\gamma +1} \cup \{ T_\gamma \}$
   and $\T_{\beta,\gamma+2} = \T_{\beta,\gamma + 1}$ for $\beta\neq\alpha_\gamma$,
\item $\T_{\alpha,\gamma +1} = \T_{\alpha,\gamma} = \bigcup_{\delta<\gamma} \T_{\alpha,\delta}$ for limit $\gamma$,
\item for all $z \in V_{\gamma + 1}$, $z\notin x_z + [T_\gamma]$.
\end{itemize}

To do this, fix $\gamma$ and work in $V_{\gamma + 2}$.
Let $f \in \twoom$ be a new real: $f \in V_{\gamma+2} \sem V_{\gamma+1}$.
By taking a ``new part'' of a canonical partition of $T_\gamma^*$ into perfectly many perfect sets, we can get $T_\gamma \leq T_\gamma^*$ such that all branches through $T_\gamma$ are new, i.e., $[T_\gamma] \cap V_{\gamma+1} = \emptyset$.
More precisely,
let
$S^f$
be the tree consisting of those $s \in \twolom$ such that
$s(i) = f(i/2)$
for all even $i < |s|$;
let~$\iota$
be the canonical
bijection between $\twolom$ and the splitting nodes of $T_\gamma^*$, and
define~$T_\gamma$ to
be
the downward closure of
$\{ \iota(s) : s \in S^f \}$; clearly,
$T_\gamma \leq T_\gamma^*$, and all branches through $T_\gamma$ are new: if there were
$y \in [T_\gamma] \cap V_{\gamma+1}$, then $f$ could be computed from
$y$
and
$T_\gamma^* \in V_{\gamma+1}$
and so $f$ would belong to $V_{\gamma+1}$ as well.
In particular,
$z \notin x_z + [T_\gamma]$ for each $z \in V_{\gamma+1}$,
i.e.,
the last clause above is satisfied.
Again, $\bar \T = \bar\T_\kappa$ is still a dense $G$-matrix.

Now, let $z \in\twoom$ and $T \in \T$. Say $T = T_\gamma \in V_{\gamma+2}$ and $z$ first arises in $V_\delta$.
If $\delta \leq \gamma + 1$, then $z \notin x_z + [T_\gamma]$ by the previous paragraph.
If $\delta \geq \gamma+2$, then $z \notin x_z + [T_\gamma^*]$ by the construction of $x_z$ according to both
second extension
lemmas
(see the
third
paragraph above). A fortiori, $z \notin x_z + [T_\gamma]$.
\end{proof}

\end{document}